\documentclass[10pt,a4paper]{amsart}
\usepackage{amsmath,amssymb,amsthm,wasysym,enumerate,url}

\numberwithin{equation}{section}

\theoremstyle{plain}
\newtheorem{theorem}[equation]{Theorem}
\newtheorem{lemma}[equation]{Lemma}

\newtheorem{proposition}[equation]{Proposition}
\newtheorem{hyp}[equation]{Hypotheses}

\newtheorem*{theorem*}{Theorem}
\newtheorem*{lemma*}{Lemma}
\newtheorem*{corollary*}{Corollary}
\newtheorem*{proposition*}{Proposition}
\newtheorem*{hyp*}{Hypotheses}

\theoremstyle{definition}
\newtheorem{defn}[equation]{Definition}
\newtheorem*{defn*}{Definition}

\theoremstyle{remark}
\newtheorem{rem}[equation]{Remark}
\newtheorem*{rem*}{Remark}

\newcommand{\C}{\ensuremath{\mathbb{C}}}

\newcommand{\N}{\ensuremath{\mathbb{N}}} 
 
\newcommand{\Z}{\ensuremath{\mathbb{Z}}}

\makeatletter
\newenvironment{proofof}[1]{\par
  \pushQED{\qed}%
  \normalfont \topsep6\p@\@plus6\p@\relax
  \trivlist
  \item[\hskip\labelsep
        \bfseries
    Proof of #1\@addpunct{.}]\ignorespaces
}{%
  \popQED\endtrivlist\@endpefalse
}
\makeatother

\begin{document}
\title{Cocycle twists of algebras}
\author{Andrew Davies}
\address{School of Mathematics\\ University of Manchester\\  Manchester\\  United Kingdom\\  M13 9PL}
\email{andrewpdavies@gmail.com}
\subjclass{16S35, 16S38, 16W22}
\keywords{Cocycle twists, Zhang twists, AS-regular}
\date{\today}

\begin{abstract}
Let $A = \bigoplus_{n=0}^{\infty}A_n$ be a connected graded $k$-algebra over an algebraically closed field $k$ (thus $A_0=k$). Assume that a finite abelian group $G$, of order coprime to the characteristic of $k$, acts on $A$ by graded automorphisms. In conjunction with a suitable cocycle this action can be used to twist the multiplication in $A$. We study this new structure and, in particular, we describe when properties like Artin-Schelter regularity are preserved by such a twist. We then apply these results to examples of Rogalski and Zhang.
\end{abstract}

\maketitle

\section{Introduction}\label{sec: introduction}

In this paper we study a twisting operation on algebras that can be formulated as a Zhang twist \cite{zhang1998twisted} or defined in the context of group actions as in \cite[\S 7.5]{montgomery1993hopf}. We will be primarily concerned with whether certain properties are preserved under such twists, following the work of Montgomery in \cite{montgomery2005algebra}. 

If an algebra $A$ is acted on by a finite abelian group $G$ then a $G$-grading is induced. We will denote a cocycle twist of this grading using a 2-cocycle $\mu$ by $A^{G,\mu}$. Such twists can be described in another manner, with the following result of Bazlov and Berenstein showing their equivalence.
\begin{proposition}[{\cite[Lemma 3.6]{bazlov2012cocycle}}]
The algebra $A^{G,\mu}$ is isomorphic to the fixed ring $(AG_{\mu})^G$ for some action of $G$ on the twisted group algebra $AG_{\mu}$.     
\end{proposition}
 
We show that applying such twists --- in particular to connected graded algebras when the action of $G$ respects this structure --- yields some interesting results. 
 
Our main theorem is as follows. It brings together Proposition \ref{prop: uninoeth}, Theorem \ref{thm: asreg} and Propositions \ref{prop: koszul} and \ref{prop: cohenmac}.
\begin{theorem}\label{thm: maintheorem}
Assume that $A$ is a noetherian connected graded algebra and $G$ a finite abelian group that acts on $A$ by graded automorphisms. If $A$ has one of the following properties then the cocycle twist $A^{G,\mu}$ shares that property:
\begin{itemize}
\item[(i)] it is strongly noetherian;
\item[(ii)] it is AS-regular;
\item[(iii)] it is Koszul;
\item[(iv)] it is Auslander regular;
\item[(v)]  it is Cohen-Macaulay.
\end{itemize}
\end{theorem}
Moreover, some of the twists we have uncovered have not been studied previously (see \cite{davies2014cocycle2}).
 
Our ability to prove that properties are preserved under twisting stems from the fact that one of the constructions of such twists has not been fully exploited (it was briefly remarked upon in \cite[\S 3.4]{bazlov2012cocycle}). This formulation generalises an example of Odesskii from \cite[pg. 89-90]{odesskii2002elliptic}, and is especially useful since it allows the use of faithful flatness arguments via the following lemma.
\begin{lemma}[{Lemma \ref{lem: fflat}}]\label{lem: fflat-intro}
As an $(A^{G,\mu},A^{G,\mu})$-bimodule there is a decomposition
\begin{equation*}
AG_{\mu} \cong \bigoplus_{g \in G} {^{\text{id}}(A^{G,\mu})^{\phi_{g}}},
\end{equation*}
for some automorphisms $\phi_{g}$ of $A^{G,\mu}$, with $\phi_e=\text{id}$. Each summand is free as a left and right $A^{G,\mu}$-module. Consequently $AG_{\mu}$ is a faithfully flat extension of both $A^{G,\mu}$ and $A$ on both the left and the right.
\end{lemma}

Let us describe Odesskii's example, which uses a 4-dimensional Sklyanin algebra. Such an object is important in noncommutative algebraic geometry in the sense of \cite{artin1990some}, whose construction can be phrased in terms of an elliptic curve and a point upon it. 

Consider a 4-dimensional Sklyanin algebra over $\C$, which we denote by $A$; its parameters and relations are unimportant for the purpose of the example. There is an action of the Klein four-group $G=(C_2)^2$ on $A$ by graded algebra automorphisms and also on $M_2(\C)$, the ring of $2 \times 2$ complex matrices. For $M \in M_2(\C)$ and generators $g_1, g_2 \in G$, the action is defined by
\begin{equation}\label{eq: matrixaction}
M^{g_{1}} =\begin{pmatrix}
-1 & 0 \\
0 & 1 
\end{pmatrix}M\begin{pmatrix}
-1 & 0 \\
0 & 1 
\end{pmatrix},\; M^{g_{2}}=\begin{pmatrix}
0 & 1 \\
1 & 0 
\end{pmatrix}M\begin{pmatrix}
0 & 1 \\
1 & 0 
\end{pmatrix}.
\end{equation}
Now take the tensor product of $\C$-algebras $A \otimes_{\C} M_2(\C)$. The example is then given by taking the invariant ring under the diagonal action of $G$, $\left(A \otimes_{\C} M_2(\C)\right)^G$. 

It is natural to wonder if the properties of $A$ are shared by this twisted algebra and if this construction can be generalised to any algebra or group. Our attempts to answer these questions motivated the work in this paper. In the subsequent paper \cite{davies2014cocycle2} we study twists of 4-dimensional Sklyanin algebras and related algebras from the perspective of noncommutative algebraic geometry. Theorem \ref{thm: maintheorem} applies to such algebras, providing new examples of AS-regular algebras of global dimension 4.

In a paper appearing on the arXiv in January 2015, Chirvasitu and Smith \cite{chirvasitu2015exotic} prove some of the same results. As noted in their paper, our results had already appeared on the internet in July 2014 \cite{davies2014thesis}.

\subsection{Contents} 
We now give a brief description of the contents of this paper. In \S\ref{sec: background} we define classical cocycle twists of group-graded algebras, while in \S\ref{sec: construction} we construct the cocycle twists that we will study in two different ways. Our main results appear in \S\ref{sec: preservation}, where we prove both Theorem \ref{thm: maintheorem} and Lemma \ref{lem: fflat-intro}.

To end the paper we look for cocycle twists in the context of Rogalski and Zhang's classfication of AS-regular algebras of dimension 4 with three generators and a proper $\Z^{2}$-grading \cite{rogalski2012regular}. We show in \S\ref{sec: rogzhang} that several of the families in their classification are related via cocycle twists.

\subsection{Notation}\label{subsec: notation}
Throughout $k$ will denote an algebraically closed field and $G$ a finite abelian group, unless otherwise stated.
Further assumptions on the characteristic of $k$ will be made at the beginning of \S\ref{sec: construction}. By $A$ we
will denote an associative $k$-algebra, often with the additional assumption that it is $\N$-graded. This means that $A
= \bigoplus_{n \in \N} A_n$ with $ab \in A_{n+m}$ for all $a \in A_n$ and $b \in A_m$. Such an algebra is said
to be \emph{connected graded} (or \emph{c.g.}\ for brevity) if $A_0=k$ and $\text{dim}_k A_n < \infty$ for all $n \in
\N$. The \emph{Hilbert series} of a c.g.\ algebra is the power series $H_A(t)=\sum_{n \in \N} (\text{dim}_k A_n)t^n$.

When describing relations in such an algebra, we will use shorthand notation for two kinds of commutator. For $x,y \in A$, define
\begin{equation*}
[x,y]:=xy-yx \text{ and } [x,y]_+ := xy+yx.
\end{equation*}

By $\text{Mod}(A)$ we shall denote the category of $A$-modules and by $\text{GrMod}(A)$ the category of $\N$-graded $A$-modules. If necessary we will specify whether these are categories of left $A$-modules or right $A$-modules. By $\text{lgldim }A$ and $\text{rgldim }A$ we will denote the left and right global dimensions of $A$ respectively, and by $\text{pdim }M_A$ ($\text{idim }M_A$) the projective (injective) dimension of a right $A$-module $M$. When $\text{lgldim }A = \text{rgldim }A$ we will write $\text{gldim }A$.

We will often consider an action of the group $G$ on $A$ by $k$-algebra automorphisms. The action of an element $g\in G$ on $a \in A$ will be denoted by the superscript $a^g$. We will denote the group of group automorphisms of $G$ by $\text{Aut}(G)$, and similarly use $\text{Aut}_{\N\text{-alg}}(A)$ to denote the group of $\N$-graded $k$-algebra automorphisms of $A$ when it is graded.

The tensor product $\otimes$ will denote the tensor product over $k$, $\otimes_k$, if no other subscript appears.

\section*{Acknowledgements}
The contents of this paper form part of the author's PhD thesis \cite{davies2014thesis}, completed under the supervision of Professor Toby Stafford. The author wishes to express their gratitude to Professor Stafford for his support and guidance throughout both the preparation of this paper and their PhD, as well as to the EPSRC for funding their study.

\section{Background}\label{sec: background}
We begin by defining cocycle twists of group-graded algebras, a classical construction that underpins the twists that we construct in \S\ref{sec: construction}.

Assume that $A$ is an associative $k$-algebra and $G$ a finite group, not necessarily abelian, such that $A$ admits a $G$-graded structure $A = \bigoplus_{g \in G} A_g$. Consider all functions $\mu: G \times G
\rightarrow k^{\times}$ satisfying the following relations for all $g,h,l \in G$:
\begin{equation}\label{eq: cocycleid}
\mu(g,h)\mu(gh,l)=\mu(g,hl)\mu(h,l),\; \mu(e,g)=\mu(g,e)=1.
\end{equation}
Such a function is called a \emph{2-cocycle of $G$ with values in $k^{\times}$} (or more formally a \emph{normalised} 2-cocycle). One can define a group structure on the set of 2-cocycles of $G$, denoted by $Z^2(G,k^{\times})$, via pointwise multiplication.

A new multiplication $\ast_{\mu}$ can be defined for all homogeneous elements $a \in A_g$ and $b \in A_h$ by 
\begin{equation*}
a \ast_{\mu} b := \mu(g,h) ab, 
\end{equation*}
where juxtaposition denotes the `old' multiplication in $A$. The new multiplication is then extended to the whole of $A$ by linearity. 
\begin{defn}\label{defn: basiccocycletwist}
With $\mu$ as above, set $A_{\mu} := (A, \ast_{\mu})$.
\end{defn}

2-cocycles are precisely those functions $G \times G \rightarrow k^{\times}$ that preserve associativity when deforming the multiplication in an algebra in this manner. Moreover, twisting by a 2-cocycle preserves the identity element of an algebra. The simplest examples of such twists are twisted group algebras, denoted by $kG_{\mu}$. 

Now consider a 2-cocycle $\phi$ for which there exists a function $\rho: G \rightarrow k^{\times}$ such that
\begin{equation*}
\phi(g,h)=\rho(g)\rho(h)\rho(gh)^{-1}, 
\end{equation*}
for all $g, h \in G$. Such a cocycle is called a \emph{2-coboundary}. As \cite[Corollary 33.6]{karpilovsky1987structure} shows in a more general situation, the twisted group algebras $kG_{\mu}$ and $kG_{\phi}$ are isomorphic as $G$-graded algebras if and only if $\mu, \phi \in Z^2(G,k^{\times})$ lie in the same coset modulo the subgroup of coboundaries, $B^2(G,k^{\times})$. Thus the isomorphism classes of $G$-graded deformations of $kG$ are parameterised by
$Z^2(G,k^{\times})/B^2(G,k^{\times})$, which is called the \emph{Schur multiplier} of $G$ \cite{karpilovsky1987schur}.

As \cite[Example 2.9]{zhang1998twisted} shows, the cocycle twists defined in Definition \ref{defn: basiccocycletwist} can be formulated as Zhang twists. For cocycle twists of $G$-graded algebras one therefore has an equivalence between their categories of $G$-graded modules by \cite[Theorem 3.1]{zhang1998twisted}. 

Finally, it will be useful for us to consider the twisted group algebra $AG_{\mu}$ as a crossed product, whose
definition appears in \cite[\S 1.5.8]{mcconnell2001noncommutative}.

\section{Constructions}\label{sec: construction}

Let us fix the base assumptions under which we will work.
\begin{hyp}[General case]\label{hyp: generalcase}
Let $A$ be a $k$-algebra where $k$ is an algebraically closed field. Assume that a finite abelian group $G$ acts on $A$ by algebra automorphisms, where $\text{char}(k) \nmid |G|$. Fix an isomorphism between $G$ and its group of characters $G^{\vee}$, mapping $g \mapsto \chi_g$.  
\end{hyp}

Our primary interest in cocycle twists is to apply them to $\N$-graded algebras. As such, we record the following additional assumptions that will often be used.
\begin{hyp}[$\N$-graded case]\label{hyp: gradedcase}
Further to Hypotheses \ref{hyp: generalcase}, assume that $A$ is $\N$-graded and $G$ acts on $A$ by $\N$-graded algebra
automorphisms, i.e. $G \rightarrow \text{Aut}_{\N\text{-alg}}(A)$.
\end{hyp}

\subsection{Cocycle twists from group actions}\label{subsec: coycletwistgroupaction}
In \S\ref{sec: background} we defined twists of a $G$-graded algebra. We will now show that when $G$ is finite
abelian a $G$-grading is induced by an action of $G$ on $A$ by algebra automorphisms. 

Assume that Hypotheses \ref{hyp: generalcase} hold. By Maschke's Theorem, $A$ splits into a direct sum of 1-dimensional irreducible $kG$-submodules. One defines a grading on $A$ by setting $A_g:=A^{\chi_{g^{-1}}}$, the isotypic component of $A$ corresponding to the character $\chi_{g^{-1}}$. To see that this defines a grading, note that for all homogeneous elements $a \in A_{g_{1}}, b \in A_{g_{2}}$ and $h \in G$,
\begin{equation*}
(ab)^h=a^hb^h=\chi_{g_{1}^{-1}}(h)a \chi_{g_{2}^{-1}}(h)b=\chi_{(g_1g_2)^{-1}}(h)ab,
\end{equation*}
which implies that $ab \in A_{g_{1}g_{2}}$.
\begin{defn}\label{defn: fullcocycletwist}
With the $G$-graded structure described above and a cocycle $\mu$, the resulting twisted algebra is written $A^{G,\mu}:= (A = \bigoplus_{g \in G} A_g, \ast_{\mu})$. Thus for all $a \in A_{g}, b \in A_{h}$ one has $a \ast_{\mu} b = \mu(g,h)ab$.
\end{defn}

We now describe another construction, once again working under Hypotheses \ref{hyp: generalcase}. Define an action of
$G$ on
the twisted group algebra $kG_{\mu}$ by $g^{h}:=\chi_g(h)g$ for all $g, h \in G$ and extending $k$-linearly. Observe
that under this action $kG_{\mu}$ is the regular representation of $G$, with isotypic components of the form
$\left(kG_{\mu}\right)^{\chi_{g}}=kg$. 

Given this action we can consider the diagonal action of $G$ on the tensor product $A \otimes kG_{\mu} = AG_{\mu}$, where $(ag)^h=a^h g^h= \chi_{g}(h)a^h g$ for all $a \in A$, $g, h \in G$. The algebra in which we are interested is the invariant ring under this action, $(A \otimes G_{\mu})^G=(AG_{\mu})^G$.

The constructions just defined are related by the following result of Bazlov and Berenstein.
\begin{proposition}[{\cite[Lemma 3.6]{bazlov2012cocycle}}]\label{prop: twoconstrequal}
Assume that Hypotheses \ref{hyp: generalcase} hold. Then $A^{G,\mu} \cong (AG_{\mu})^G$ as $k$-algebras.
\end{proposition}

Henceforth, our use of the term \emph{cocycle twist} and the notation $A^{G,\mu}$ will refer to either of the equivalent twists in this proposition. 

Odesskii's example from the introduction is an illustration of the invariant ring construction of a cocycle twist. The only subtlety is that the ring $M_2(\C)$ is in fact isomorphic to a twisted group algebra over the Klein four-group for a nontrivial 2-cocycle. 

\subsection{Twisting the $G$-grading}\label{subsec: twistggrading}
In this section we investigate the effect of twisting a $G$-grading by a group automorphism. This idea is described in \cite[Example 3.8]{zhang1998twisted}. Given a $G$-graded algebra $A$ and $\sigma \in \text{Aut}(G)$, one can define a new grading on $A$ by $A_{\sigma}=\bigoplus_{g \in G} B_g$ where $B_g :=A_{\sigma(g)}$ for all $g \in G$. When $G$ is abelian this grading corresponds to another action of $G$ on $A$ by $k$-algebra automorphisms. 

We wish to connect this new grading with a cocycle twist. Given a cocycle $\mu$ we can define an action of $\sigma$ on $\mu$ by 
\begin{equation*}
\mu^{\sigma}(g,h):=\mu(\sigma(g),\sigma(h)),
\end{equation*}
for all $g,h \in G$. It is clear that this is also a 2-cocycle. Moreover, the action of $\sigma$ preserves
2-coboundaries and so there is an action of $\text{Aut}(G)$ on the Schur multiplier of $G$.
%
\begin{lemma}\label{lem: autoncocycle}
For a 2-cocycle $\mu$ the cocycle twist $(A_{\sigma},\ast_{\mu})$ is isomorphic as a $k$-algebra to $(A,\ast_{\mu^{\left(\sigma^{-1}\right)}})$.
\end{lemma}
\begin{proof}
In the twist $(A_{\sigma},\ast_{\mu})$ consider homogeneous elements $a \in B_g$ and $b \in
B_h$. Under the graded structure in $(A,\ast_{\mu^{\left(\sigma^{-1}\right)}})$ one has $a \in A_{\sigma(g)}$ and $b \in
A_{\sigma(h)}$. Writing the multiplication of $a$ and $b$ in $(A_{\sigma},\ast_{\mu})$ gives
\begin{equation}\label{eq: autactoncocyclemult}
a \ast_{\mu} b = \mu(g,h)ab = \mu^{(\sigma^{-1})}(\sigma(g),\sigma(h))ab.
\end{equation}
Notice that the right-hand side of \eqref{eq: autactoncocyclemult} is precisely the multiplication $a \ast_{\mu^{\left(\sigma^{-1}\right)}} b$ in $(A,\ast_{\mu^{\left(\sigma^{-1}\right)}})$.
\end{proof}

We now examine the choice of isomorphism $G \rightarrow G^{\vee}$. Let us use the notation $(A,\phi,\mu)$ for a triple consisting of an algebra, an isomorphism $G \rightarrow G^{\vee}$, and a 2-cocycle respectively. When $G$ acts on $A$ by algebra automorphisms each such triple can be naturally associated to a cocycle twist.

\begin{proposition}\label{prop: benign}
Let $G$ act on $A$ by algebra automorphisms. Let $\phi$ and $\rho$ be isomorphisms $G \rightarrow G^{\vee}$ and $\mu$ be
a 2-cocycle. Then there exists an automorphism of $G$, $\tau$ say, such that the cocycle twists corresponding to the
triples $(A,\phi,\mu)$ and $(A,\rho,\mu^{(\tau^{-1})})$ are isomorphic as $k$-algebras.
\end{proposition}
\begin{proof}
Given $\phi$, we will identify $\rho$ with an automorphism of $G$ as follows. Firstly, there exists an automorphism $\psi: G^{\vee} \rightarrow G^{\vee}$ such that $\phi = \psi \circ \rho$. Suppose we have an element $x \in A_g$, where the grading is determined under the duality given by $\phi$. This means that for all $h \in G$,
\begin{equation*}
x^h = \phi(g)^{-1}(h)x=\psi(\rho(g))^{-1}(h)x.
\end{equation*}

Since all maps involved are isomorphisms, for all $g \in G$ there exists $k_g \in G$ such that $\psi(\rho(g))=\rho(k_g)$. We claim that the map $\tau:\; g \mapsto k_g$ defines an isomorphism of $G$. To see this, note that for all $g,h \in G$ one has
\begin{equation*}
\rho(k_{gh})=\psi(\rho(gh))=\psi(\rho(g)) \psi(\rho(h)) =\rho(k_{g})\rho(k_{h})=\rho(k_{g}k_{h}).
\end{equation*}

As $\rho$ is an isomorphism, it follows that $\tau$ is also an isomorphism as claimed. Under
the duality isomorphism $\rho$ one has $x \in A_{k_{g}}$, since 
\begin{equation}\label{eq: dualgrading1}
x^h =\psi(\rho(g))^{-1}(h)x=\rho(k_g)^{-1}(h)x,
\end{equation}
for all $h \in G$.

Suppose that $x \in A_g$ and $y \in A_h$ for some $g,h \in G$ under the duality given by $\phi$. Thus $x
\ast_{\mu} y=\mu(g,h)xy$ in $(A,\phi,\mu)$. Under the duality given by $\rho$ one has $x \in A_{k_{g}}$ by \eqref{eq:
dualgrading1}, and in a similar manner $y \in A_{k_{h}}$. Thus in $(A,\rho,\mu^{(\tau^{-1})})$ the multiplication is
\begin{equation*}
x \ast_{\mu^{\left(\tau^{-1}\right)}} y = \mu^{(\tau^{-1})}(k_g,k_h)xy=\mu(\tau^{-1}(k_g),\tau^{-1}(k_h))xy=\mu(g,h)xy. 
\end{equation*}
Since the multiplications agree on homogeneous elements, this completes the proof.
\end{proof}

\section{Preservation of properties}\label{sec: preservation}
In this section we prove that many properties are preserved by the twists defined in \S\ref{sec: construction}.

\subsection{Basic properties}\label{subsec: basicprops} 
\emph{Unless otherwise stated we assume that Hypotheses \ref{hyp: generalcase} hold for all results in this section.}

We first state a useful result regarding the behaviour of regular and normal elements under a cocycle twist. This result is not stated explicitly in \cite{zhang1998twisted}, although the proof is essentially contained in that of Proposition 2.2(1) op. cit.. 
\begin{lemma}\label{lem: stillregular}
Any element $a \in A$ that is homogeneous with respect to the $G$-grading is regular (normal) in $A$ if and only if it is regular (normal) in $A^{G,\mu}$.
\end{lemma}

The next two lemmas (the latter from \cite{montgomery2005algebra}) will be particularly useful when working with algebras defined by generators and relations.
\begin{lemma}\label{lem: defrelns}
Let $I$ be a $G$-graded ideal of $A$. Then $I$ remains an ideal in $A^{G,\mu}$. Furthermore, a generating set for $I$ that is homogeneous with respect to the $G$-grading is also a generating set for the ideal under twisting.
\end{lemma}
\begin{proof}
That $I$ is still an ideal in the twist is proved in \cite[Proposition 3.1(2)]{montgomery2005algebra}. To complete the proof it suffices to deal with the case that $I=(f)$ for some homogeneous element $f \in A_g$. Suppose that $a \in A$ with homogeneous decomposition $a = \sum_{h \in G} a_h$. One has
\begin{equation*}
fa = f \ast_{\mu} \left(\sum_{h \in G} \frac{a_h}{\mu(g,h)} \right)\; \text{ and }\; af = \left(\sum_{h \in G} \frac{a_h}{\mu(h,g)} \right) \ast_{\mu} f,
\end{equation*}
which proves the result.
\end{proof}
\begin{lemma}[{\cite[ Proposition 3.1(1)]{montgomery2005algebra}}]\label{lemma: finitelygenerated}
$A$ is finitely generated  as a $k$-algebra if and only if $A^{G,\mu}$ is also finitely generated. Furthermore, if Hypotheses \ref{hyp: gradedcase} hold then $A$ is finitely generated in degree 1 if and only if $A^{G,\mu}$ is too.
\end{lemma}
\begin{proof}
The first part of the statement is proved by Montgomery. By consulting the proof in \cite{montgomery2005algebra}, one can see that a generating set for $A^{G,\mu}$ can be obtained as follows: take a generating set of $A$ and find a vector space $V$ which contains this generating set and is preserved by the action of $G$. Then $A^{G,\mu}$ will be generated by $V$ under the new multiplication on the shared underlying vector space. 

One may therefore conclude that under the additional hypotheses of the second statement of this lemma, the property of being finitely generated in degree 1 is preserved.
\end{proof}

We now show that gradings are sometimes preserved under cocycle twists.
\begin{lemma}\label{lem: autpresgrad}
Suppose that $A$ has a $H$-grading for some arbitrary group $H$ and that a finite abelian group $G$ acts on $A$ by $H$-graded algebra automorphisms. Then any cocycle twist $A^{G,\mu}$ will inherit the $H$-grading from $A$.
\end{lemma}
\begin{proof}
We must show that for all $h_1,h_2 \in H$ and homogeneous elements $x \in A_{h_{1}}$ and $y \in A_{h_{2}}$ one has $x \ast_{\mu} y \in A_{h_{1}h_{2}}$, since then $A^{G,\mu}_{h_{1}} \cdot A^{G,\mu}_{h_{2}} \subseteq A^{G,\mu}_{h_{1}h_{2}}$. As $G$ acts on $A$ by $H$-graded algebra automorphisms, one can apply Maschke's theorem to the $H$-graded components of $A$ under the action of $G$. This allows us to further assume that $x$ and $y$ are homogeneous with respect to the $G$-grading, thus $x \in A_{g_{1}}$ and $y \in A_{g_{2}}$ for some $g_1,g_2 \in G$. Then
\begin{equation*}
x \ast_{\mu} y = \mu(g_1,g_2)xy \in A_{h_{1}h_{2}},
\end{equation*}
which completes the proof.
\end{proof}
\begin{rem}
In particular, Lemma \ref{lem: autpresgrad} implies that $A^{G,\mu}$ inherits the $G$-grading from $A$. 
\end{rem}

Before stating our next result, we recall the concept of twisting a module by an automorphism. Let $A$ be a $k$-algebra and $\phi$ be a $k$-algebra automorphism. For a right $A$-module $M$, one can define a new right $A$-module $M^{\phi}$ via the multiplication $m \ast_{\phi} a=m\phi(a)$ for all $a \in A$, $m \in M$. One can twist both sides of an $(A,A)$-bimodule in this manner simultaneously; in particular, if such an $(A,A)$-bimodule is free on each side and the same generator can be used in each module structure, then one may assume that the bimodule is untwisted on one side (see \cite[\S2.3]{brown2008dualising}).
\begin{lemma}\label{lem: fflat}
As an $(A^{G,\mu},A^{G,\mu})$-bimodule there is a decomposition
\begin{equation*}
AG_{\mu} \cong \bigoplus_{g \in G} {^{\text{id}}(A^{G,\mu})^{\phi_{g}}},
\end{equation*}
for some automorphisms $\phi_{g}$ of $A^{G,\mu}$, with $\phi_e=\text{id}$. Each summand is free of rank 1 as a left and right $A^{G,\mu}$-module. Consequently, $AG_{\mu}$ is a faithfully flat extension of $A^{G,\mu}$ on both the left and the right. 

Similarly, $_A(AG_{\mu})$ and $(AG_{\mu})_A$ are free modules of finite rank, thus $AG_{\mu}$ is a faithfully flat extension of $A$ on both the left and the right. 
\end{lemma}
\begin{proof}
We will proceed as in the proof of the main theorem of \cite{smith1989can}. Let $AG_{\mu}=\bigoplus_{g \in G}
M^{\chi_{g}}$ be the isotypic decomposition of $AG_{\mu}$ under the action
of $G$. Observe that $A^{G,\mu}=M^{\chi_{e}}$ and $M^{\chi_{g}}M^{\chi_{h}}=M^{\chi_{gh}}$ for all $g,h \in G$, since $G$ acts by algebra automorphisms. This means that each isotypic component $M^{\chi_{g}}$ has an $(A^{G,\mu},A^{G,\mu})$-bimodule
structure. 

The isotypic component $M^{\chi_{g}}$ contains the element $1 \otimes g$. An arbitrary element in this component has the form $a \otimes h$ for some $a \in A_{g^{-1}h}=A^{\chi_{gh^{-1}}}$. Thus $a \otimes g^{-1}h \in A^{G,\mu}$ and therefore
\begin{equation*}
a \otimes h= \left(\frac{a \otimes g^{-1}h}{\mu(g^{-1}h,g)}\right) \cdot (1 \otimes g)=(1 \otimes g) \cdot \left(\frac{a \otimes g^{-1}h}{\mu(g,g^{-1}h)}\right).
\end{equation*}
Consequently, $M^{\chi_{g}}$ is cyclic as a left or a right $A^{G,\mu}$-module. Note that $1 \otimes g$ is regular in $AG$, therefore by Lemma \ref{lem: stillregular} it is also regular in $AG_{\mu}$. This proves that $M^{\chi_{g}}$ is a free $A^{G,\mu}$-module of rank 1 on both the left and the right.

By the discussion prior to the statement of the proposition, we know that the bimodule generated by $1 \otimes g$ is isomorphic to ${^{\text{id}}}(A^{G,\mu})^{\phi_{g}}$ for some algebra automorphism $\phi_{g}$. To describe $\phi_g$ it suffices to look at the left action of a homogeneous element in $A^{G,\mu}$ on $1 \otimes g$, which can be taken to be a free generator for the left $A^{G,\mu}$-module structure. Consider a homogeneous element $a \otimes h \in A^{G,\mu}_h$. One has
\begin{equation*}
(a \otimes h) \cdot (1 \otimes g) = \mu(h,g) a \otimes hg = (1 \otimes g) \cdot  \frac{\mu(h,g)}{\mu(g,h)}(a \otimes h).
\end{equation*}

Define a map $\phi_g: A^{G,\mu} \rightarrow A^{G,\mu}$ by $a \otimes h \mapsto \frac{\mu(h,g)}{\mu(g,h)}(a \otimes h)$ on homogeneous elements and extending $k$-linearly. To see that this is a $G$-graded automorphism, consider homogeneous elements $a \otimes h \in A^{G,\mu}_h$ and $b \otimes l \in A^{G,\mu}_l$. Then
\begin{equation}\label{eq: phighom}
\phi_g(a \otimes h)\phi_g(b \otimes l) = \frac{\mu(h,g)\mu(l,g)\mu(h,l)}{\mu(g,h)\mu(g,l)}(ab\otimes hl).
\end{equation}

On the other hand, one can use \eqref{eq: cocycleid} to see that
\begin{gather}
\begin{aligned}\label{eq: phighom1}
\phi_g(\mu(h,l)(ab\otimes hl)) &= \frac{\mu(hl,g)\mu(h,l)}{\mu(g,hl)} (ab\otimes hl) 
\\ &= \frac{\mu(h,lg)\mu(l,g)\mu(h,l)}{\mu(g,h)\mu(gh,l)} (ab\otimes hl).
\end{aligned}
\end{gather}
Observe that $\frac{\mu(h,lg)}{\mu(gh,l)} = \frac{\mu(h,g)}{\mu(g,l)}$, which follows from $G$ being abelian together with another use of \eqref{eq: cocycleid}. Substituting this expression into \eqref{eq: phighom1} produces the expression in \eqref{eq: phighom}. It is clear that $\phi_g$ is injective, therefore it must be a $G$-graded automorphism of $A^{G,\mu}$ as claimed.

The result is trivial for $A$ by the definition of $AG_{\mu}$.
\end{proof}

The previous result allows us to begin proving that various properties are preserved under twisting, beginning with GK dimension. Part (ii) of the following lemma is implicit in \cite[pg. 89-90]{odesskii2002elliptic}.
\begin{lemma}\label{lem: hilbseries}
The following statements are true:
\begin{itemize}
 \item[(i)] $\text{GKdim }A= \text{GKdim }A^{G,\mu}$;
 \item[(ii)] Under Hypotheses \ref{hyp: gradedcase} one has $H_A(t)=H_{A^{G,\mu}}(t)$. In particular, if $A$ is connected graded then so is $A^{G,\mu}$.
\end{itemize}
\end{lemma}
\begin{proof}
By Lemma \ref{lem: fflat}, $AG_{\mu}$ is a f.g.\ module over $A$ and $A^{G,\mu}$ on both sides. Applying \cite[Proposition 5.5]{krause2000growth} twice, first to $A \subset AG_{\mu}$, then to $A^{G,\mu} \subset AG_{\mu}$, proves part (i) of the lemma.

Now assume that Hypotheses \ref{hyp: gradedcase} hold. By Lemma \ref{lem: autpresgrad}, $A^{G,\mu}$ possesses the same $\N$-graded structure as $A$, thus the dimensions of the graded components are the same for both.
\end{proof}

We now turn to the strongly noetherian property which, as demonstrated by results in \cite{rogalski2008canonical}, has strong geometric consequences. 
\begin{defn}[{\cite[cf. \S 4]{artin1999generic}}] \label{defn: strongnoeth}
Let $A$ be a noetherian $k$-algebra. $A$ is \emph{strongly (right) noetherian} if for all commutative noetherian $k$-algebras $R$, $A \otimes R$ is (right) noetherian. 
\end{defn}

Our result is a partial generalisation of \cite[Proposition 3.1(3)]{montgomery2005algebra} which was concerned with the noetherian condition.
\begin{proposition}\label{prop: uninoeth}
$A$ is strongly noetherian if and only if $A^{G,\mu}$ is.
\end{proposition}
\begin{proof}
We prove that $A^{G,\mu}$ is strongly right noetherian. Assume that $A$ is strongly noetherian. Then
$AG_{\mu}$ is a f.g.\ right $A$-module by the proof of Lemma \ref{lem: fflat}, hence by \cite[Proposition 4.1(1a)]{artin1999generic} $AG_{\mu}$ is strongly right noetherian. Using Lemma \ref{lem: fflat} again, the extension $A^{G,\mu} \subset AG_{\mu}$ is faithfully flat on the right, therefore we can apply \cite[Proposition 4.1(2a)]{artin1999generic} to show that $A^{G, \mu}$ is strongly right noetherian.

In the other direction we can use the same argument but with $AG_{\mu}$ replaced with $A^{G,\mu}G_{\mu^{-1}}$; Lemma \ref{lem: fflat} tells us that $A^{G,\mu}G_{\mu^{-1}}$ is a right faithfully flat extension of both $A^{G,\mu}$ and of $A \cong (A^{G,\mu}G_{\mu^{-1}})^G$.

An identical proof works for the left-sided part of the result.
\end{proof}

\subsection{AS-regularity}\label{subsec: asregular}
The aim of this section is to prove that being AS-regular is preserved under cocycle twists. Although this property only concerns c.g.\ algebras, we will prove that certain components of the definition are preserved without any assumption on graded structure.

\begin{defn}[{\cite{artin1987graded}}]\label{defn: asregular}
Let $d \in \mathbb{N}$. A c.g.\ $k$-algebra $A$ is said to be \emph{AS-regular of dimension $d$} if the following conditions are satisfied:
\begin{itemize}
 \item[(i)] $A$ has finite GK dimension;
 \item[(ii)] $\text{gldim }A=d$;
 \item[(iii)] $A$ is AS-Gorenstein, thus $\text{Ext}^i_A(k,A)= k \delta_{i,d}$ when $k$ and $A$ are considered as left or as right $\N$-graded $A$-modules.
\end{itemize}
\end{defn}
One can calculate the Ext group in (iii) in either $\text{Mod}(A)$ or in $\text{GrMod}(A)$; since $A$ and $k$ are f.g.\ modules, the discussion in \cite[\S 1.4]{levasseur1992some} shows that the two Ext groups in question are the same.

Let us first consider condition (ii) regarding global dimension. For the purposes of this section, we only need to show that finite global dimension is preserved under cocycle twists. However, we will prove the more general result that left and right global dimension are preserved, regardless of whether they are equal or not. 

We will need the following technical result to compare the global dimension of $A^{G,\mu}$ with that of the twisted group algebra $AG_{\mu}$. There is an analogous version for left modules.
\begin{theorem}[{\cite[Theorem 7.2.8(i)]{mcconnell2001noncommutative}}]\label{theorem: mcrobtechnical}
Let $R,S$ be rings with $R \subseteq S$ such that $R$ is an $(R,R)$-bimodule direct summand of $S$. Then
\begin{equation*}
\text{rgldim }R \leq \text{rgldim }S +\text{pdim }S_R. 
\end{equation*}
\end{theorem}

Without further ado, we now show that left and right global dimension are preserved under cocycle twists.
\begin{proposition}\label{prop: gldim}
One has $\text{rgldim }A=\text{rgldim }A^{G,\mu}$ and $\text{lgldim }A=\text{lgldim }A^{G,\mu}$.
\end{proposition}
\begin{proof}
We will give the proof for right global dimension, from which a left-sided proof can easily be derived. Since $AG_{\mu}$
has the structure of a crossed product we can apply \cite[Theorem 7.5.6(iii)]{mcconnell2001noncommutative} to
conclude that $\text{rgldim }A=\text{rgldim }AG_{\mu}$. By Lemma \ref{lem: fflat} we know that $A^{G,\mu}$ is an
$(A^{G,\mu},A^{G,\mu})$-bimodule direct summand of $AG_{\mu}$. We may therefore apply Theorem \ref{theorem: mcrobtechnical},
which tells us that
\begin{equation*}
\text{rgldim }A^{G,\mu} \leq \text{rgldim }AG_{\mu} + \text{pdim }(AG_{\mu})_{A^{G,\mu}}= \text{rgldim }AG_{\mu}.
\end{equation*}

Here $\text{pdim }(AG_{\mu})_{A^{G,\mu}}=0$ since the module is free by Lemma \ref{lem: fflat}. We have proved that $\text{rgldim }A^{G,\mu} \leq \text{rgldim }A$.

Now we repeat the argument with the roles of $A$ and $A^{G,\mu}$ reversed, considering them as subrings of $A^{G,\mu}G_{\mu^{-1}}$. One obtains the opposite inequality, which completes the proof.
\end{proof}

Let us now address the AS-Gorenstein property. Our main tool to prove the preservation of this property will be the following result of Brown and Levasseur.
\begin{proposition}[{\cite[Proposition 1.6]{brown1985cohomology}}]\label{prop: brownlevass}
Let $R$ and $S$ be rings and $R \rightarrow S$ a ring homomorphism such that $S$ is a flat as a left and right
$R$-module. Let $X$ be an $(R,R)$-bimodule such that the $(R,S)$-bimodule $X \otimes_R S$ is an
$(S,S)$-bimodule. Then for every f.g.\ left $R$-module $M$ and all $i \geq 0$, there are isomorphisms
of right $S$-modules,
\begin{equation*}
\text{Ext}_R^i(M,X)\otimes_R S \cong \text{Ext}_S^i(S \otimes_R M,X \otimes_R S). 
\end{equation*}
\end{proposition}
\begin{rem}\label{rem: x=s}
For some of our applications of Proposition \ref{prop: brownlevass} we will take $R=X$. In that case $X \otimes_R S \cong S$ as an $(R,S)$-bimodule, from which $X \otimes_R S$ inherits a natural $(S,S)$-bimodule structure. 
\end{rem}

\begin{proposition}\label{prop: asgor}
Assume that Hypotheses \ref{hyp: gradedcase} hold and that $A$ is connected graded. Then $A$ is AS-Gorenstein of global dimension $d$ if and only if $A^{G,\mu}$ shares this property.
\end{proposition}
\begin{proof}
We will give the proof in the only if direction when $k$ and $A$ are considered as left $A$-modules. The proof in the opposite direction is identical by untwisting, while the proof for right modules is almost identical to that below; the only difference is that it requires the use of a right-handed version of Proposition \ref{prop: brownlevass}.

First, observe that under the $(\N,G)$-bigrading on $AG_{\mu}$ the subalgebra consisting of elements that have degree zero under the \N-grading is the twisted group algebra $kG_{\mu}$. As right $AG_{\mu}$-modules one has isomorphisms
\begin{equation}\label{eq: degree1factor}
k \otimes_A AG_{\mu} \cong k \otimes_{A^{G,\mu}} AG_{\mu} \cong kG_{\mu}, 
\end{equation}
and as left $AG_{\mu}$-modules we have
\begin{equation}\label{eq: degree1factor1}
AG_{\mu} \otimes_A k \cong AG_{\mu} \otimes_{A^{G,\mu}} k \cong kG_{\mu}.
\end{equation}

We now proceed by applying Proposition \ref{prop: brownlevass} with $R=X=A$, $S=AG_{\mu}$ and $M=k$. To see that the hypotheses of that result are satisfied, observe that $A \subset AG_{\mu}$ is flat by Lemma \ref{lem: fflat} and recall Remark \ref{rem: x=s}. Applying Proposition \ref{prop: brownlevass} gives
\begin{gather}
\begin{aligned}\label{eq: asgorfirststep}
\text{Ext}^i_A(k,A) \otimes_A AG_{\mu} &\cong \text{Ext}^i_{AG_{\mu}}(AG_{\mu} \otimes_A k,A \otimes_A AG_{\mu}) 
\\ &\cong  \text{Ext}^i_{AG_{\mu}}(kG_{\mu},AG_{\mu}),
\end{aligned} 
\end{gather}
by using \eqref{eq: degree1factor1}. Since $A$ is AS-Gorenstein of global dimension $d$ we know that the left hand side is non-zero only for $i=d$. In that case it is equal to $k \otimes_A AG_{\mu} \cong kG_{\mu}$ by using \eqref{eq: degree1factor}. 

We would now like to apply Proposition \ref{prop: brownlevass} a second time, using $R=X=A^{G,\mu}$, $S=AG_{\mu}$ and $M=k$. We may apply the same argument as used earlier in the proof, mutatis mutandis, to see that the hypotheses of that result are satisfied. Applying Proposition \ref{prop: brownlevass} we obtain
\begin{gather}
\begin{aligned}\label{eq: asgorsecondstep}
\text{Ext}^i_{A^{G,\mu}}(k,A^{G,\mu}) \otimes_{A^{G,\mu}} AG_{\mu} & \cong  \text{Ext}^i_{AG_{\mu}}(AG_{\mu} \otimes_{A^{G,\mu}} k, A^{G,\mu} \otimes_{A^{G,\mu}} AG_{\mu}) \\
   & \cong  \text{Ext}^i_{AG_{\mu}}(kG_{\mu}, AG_{\mu}).
\end{aligned}
\end{gather}
Combining the information from \eqref{eq: asgorfirststep} and \eqref{eq: asgorsecondstep} gives
\begin{equation*}
\text{Ext}^i_{A^{G,\mu}}(k,A^{G,\mu}) \otimes_{A^{G,\mu}} AG_{\mu} \cong \left\{ \begin{array}{cl} kG_{\mu} & \text{if }i=d, \\ 0 & \text{if }i \neq
d. \end{array}\right. 
\end{equation*}

As $A^{G,\mu} \subset AG_{\mu}$ is a faithfully flat extension on the left, $\text{Ext}^i_{A^{G,\mu}}(k,A^{G,\mu})$ must vanish in all degrees for which $i \neq d$. When $i=d$ we have
\begin{equation}\label{eq: asgorfinalstep}
\text{Ext}^d_{A^{G,\mu}}(k,A^{G,\mu}) \otimes_{A^{G,\mu}} AG_{\mu} \cong kG_{\mu}, 
\end{equation}
as right $AG_{\mu}$-modules. 

Since $k$ and $A^{G,\mu}$ are f.g.\ $\N$-graded left $A^{G,\mu}$-modules, the module structure defined in
\cite[Theorem 1.15]{rotman2008introduction} implies that $\text{Ext}^i_{A^{G,\mu}}(k,A^{G,\mu})$ is a $\Z$-graded group.
This $\Z$-grading is compatible with the right $A^{G,\mu}$-module structure, in which case the graded module structure
on $\text{Ext}^i_{A^{G,\mu}}(k,A^{G,\mu})$ allows us to complete the proof as follows. Consider the
$(A^{G,\mu},A^{G,\mu})$-bimodule structure of $AG_{\mu}$ described in Lemma \ref{lem: fflat}. Upon restricting the isomorphism in \eqref{eq: asgorfinalstep} to $A^{G,\mu}$, one obtains
\begin{equation*}
\bigoplus_{g\in G} \text{Ext}^d_{A^{G,\mu}}(k,A^{G,\mu})^{\phi_{g}} \cong (kG_{\mu})_{A^{G,\mu}}.
\end{equation*}
By considering the $G$-graded components of this isomorphism and noting that $\phi_e$ is the identity, one obtains the isomorphism of right $A^{G,\mu}$-modules $\text{Ext}^d_{A^{G,\mu}}(k,A^{G,\mu}) \cong k$, which proves the result.
\end{proof}

We can now combine several previous results to prove the main theorem of this section.
\begin{theorem}\label{thm: asreg}
Assume that Hypotheses \ref{hyp: gradedcase} hold and $A$ is connected graded. Then $A$ is AS-regular if and only if $A^{G,\mu}$ is. If in addition $A$ has global and GK dimension $\leq 4$, then $A$ is a domain if and only if $A^{G,\mu}$ is a domain.
\end{theorem}
\begin{proof}
The statement about AS-regularity follows from Lemma \ref{lem: hilbseries} and Propositions \ref{prop: gldim} and \ref{prop: asgor}. The second part of the theorem follows from \cite[Theorem 3.9]{artin1991modules}.  
\end{proof}
The property of being a domain is \emph{not} preserved by Zhang twists. Such twists preserve this property when $G$ is an ordered semigroup \cite[Proposition 5.2]{zhang1998twisted} but not in general. See \cite{davies2014cocycle2} for further examples.

\subsection{The Koszul property}\label{subsec: koszul}
Our next aim is to show that the Koszul property is preserved under cocycle twists. We begin by giving one of several equivalent definitions for the property.

\begin{defn}[{\cite[ Proposition 2.1.3]{koszul1996beilinson}}]\label{defn: koszulcomplex} 
A c.g.\ $k$-algebra $A$ is \emph{Koszul} if and only if for all $i \geq 0$ the $\Z$-graded components of
$\text{Ext}_A^i(k,k)$ vanish in all degrees other than degree $i$.
\end{defn}

Let us now prove our result concerning preservation of the Koszul property.
\begin{proposition}\label{prop: koszul}
In addition to Hypotheses \ref{hyp: gradedcase}, assume that $A$ is connected graded and its defining relations are in degree 2. Then
$A$ is Koszul if and only $A^{G,\mu}$ is.
\end{proposition}
\begin{proof}
We wish to apply Proposition \ref{prop: brownlevass} with $R=A$, $S=AG_{\mu}$, $X={_A}k_A$ and $M={_A}k$. Let us check that the hypotheses are satisfied: observe that $A \subset AG_{\mu}$ is flat by Lemma \ref{lem: fflat}, while $X \otimes_R S = kG_\mu$ by \eqref{eq: degree1factor}, whence it has a natural $(AG_{\mu},AG_{\mu})$-bimodule structure. We may therefore apply Proposition \ref{prop: brownlevass}, in which case one has
\begin{gather}
\begin{aligned}\label{eq: koszulbrown}
\text{Ext}_A^i(k,k)\otimes_A AG_{\mu} &\cong \text{Ext}_{AG_{\mu}}^i(AG_{\mu} \otimes_A k,k \otimes_A AG_{\mu}) \\
 &\cong \text{Ext}_{AG_{\mu}}^i(kG_{\mu},kG_{\mu}), 
\end{aligned}
\end{gather}
using \eqref{eq: degree1factor} and \eqref{eq: degree1factor1} to pass from the first line of this equation to the second.

Now set $R= A^{G,\mu}$, $S=AG_{\mu}$, $X={_A}k_A$ and $M={_A}k$. One can use the same argument as earlier in the proof,
mutatis mutandis, to see that the hypotheses of Proposition \ref{prop: brownlevass} are satisfied. Applying that
result we obtain 
\begin{gather}
\begin{aligned}\label{eq: koszulbrown1}
\text{Ext}_{A^{G,\mu}}^i(k,k)\otimes_{A^{G,\mu}} AG_{\mu} &\cong \text{Ext}_{AG_{\mu}}^i(AG_{\mu} \otimes_{A^{G,\mu}} k,k \otimes_{A^{G,\mu}} AG_{\mu}) \\
 &\cong \text{Ext}_{AG_{\mu}}^i(kG_{\mu},kG_{\mu}), 
\end{aligned}
\end{gather}
using \eqref{eq: degree1factor} and \eqref{eq: degree1factor1} once again. 

The $\Z$-grading on $\text{Ext}_A^i(k,k)$ and $\text{Ext}_{A^{G,\mu}}^i(k,k)$ is compatible with their right $A$- and $A^{G,\mu}$-module structures respectively. Thus the tensor products $\text{Ext}_A^i(k,k)\otimes_A AG_{\mu}$ and $\text{Ext}_{A^{G,\mu}}^i(k,k)\otimes_{A^{G,\mu}} AG_{\mu}$ are naturally $\Z$-graded right $AG_{\mu}$-modules. The $\Z$-grading on the cohomology group $\text{Ext}_{AG_{\mu}}^i(kG_{\mu},kG_{\mu})$ is also compatible with its right $AG_{\mu}$-module structure. Moreover, one can see from the proof of Proposition \ref{prop: brownlevass} (in \cite[Proposition 1.6]{brown1985cohomology}) that the isomorphisms in \eqref{eq: koszulbrown} and \eqref{eq: koszulbrown1} respect these $\Z$-graded structures. We may therefore conclude that there is an isomorphism 
\begin{equation}\label{eq: comparedims}
\text{Ext}_A^i(k,k)\otimes_A AG_{\mu} \cong \text{Ext}_{A^{G,\mu}}^i(k,k)\otimes_{A^{G,\mu}} AG_{\mu} 
\end{equation}
of $\Z$-graded right $AG_{\mu}$-modules. 

Using the free module structures of $_A(AG_{\mu})$ and  $_{A^{G,\mu}}(AG_{\mu})$ described in Lemma \ref{lem: fflat}, we may express the isomorphism in \eqref{eq: comparedims} as
\begin{equation}\label{eq: comparedims1}
\bigoplus_{|G|}\text{Ext}_A^i(k,k) \cong \bigoplus_{|G|}\text{Ext}_{A^{G,\mu}}^i(k,k), 
\end{equation}
at the level of vector spaces. Furthermore, as $AG_{\mu}$ is an $\N$-graded left module over $A$ and over $A^{G,\mu}$, the isomorphism in \eqref{eq: comparedims1} respects the $\Z$-graded structure. 

Since $A$ is Koszul, we know that the $\Z$-graded components of the left hand side of \eqref{eq: comparedims1} vanish in all degrees other than degree $i$. It follows that $\text{Ext}_{A^{G,\mu}}^i(k,k)$ must also vanish in all degrees other than degree $i$, hence $A^{G,\mu}$ must be Koszul.
\end{proof}

\subsection{The Cohen-Macaulay property and Auslander regularity}\label{subsec: cohenmac} 
In this section we will prove that several more homological properties of algebras are preserved under cocycle twists.

The definitions that follow can all be found in \cite[\S 1.2]{levasseur1993modules}. The \emph{grade} of a finitely generated left or right $A$-module $M$ is defined to be the value
\begin{equation*}
j_A(M)=\text{inf}\{i: \text{Ext}_A^i(M,A)\neq 0\} \in \N \cup \{+\infty\}. 
\end{equation*}
\begin{defn}\label{def: cm}
A ring $A$ is said to satisfy the \emph{Cohen-Macaulay property} or be \emph{Cohen-Macaulay} (CM) if for all non-zero finitely generated $A$-modules $M$, one has
\begin{equation*}
\text{GKdim }M+j_A(M)=\text{GKdim }A.
\end{equation*} 
\end{defn}
\begin{defn}\label{def: auslanderprops}
The ring $A$ satistfies the \emph{Auslander-Gorenstein condition} if for every finitely generated left or right module $M$, all $i \geq 0$ and every $A$-submodule $N$ of $\text{Ext}^i_A(M,A)$, one has $j_A(N)\geq i$.

The ring is \emph{Auslander-Gorenstein} if it satisfies the Auslander-Gorenstein condition and it has finite left and right injective dimension. It is said to be \emph{Auslander regular} if in addition to satisfying the Auslander-Gorenstein condition it has finite global dimension. 
\end{defn}

The following result shows that all of the properties defined above are preserved under a cocycle twist.
\begin{proposition}\label{prop: cohenmac}
In addition to Hypotheses \ref{hyp: generalcase}, assume that $A$ is noetherian. Then $A$ has one of the following
properties if and only if $A^{G,\mu}$ does as well:
\begin{itemize}
 \item[(i)] it is Cohen-Macaulay;
 \item[(ii)] it is Auslander-Gorenstein;
 \item[(iii)] it is Auslander regular.
\end{itemize}
\end{proposition}
\begin{proof}
(i) Assume that $A$ is Cohen-Macaulay. As we proved in Lemma
\ref{lem: hilbseries}(i), $\text{GKdim }A=\text{GKdim }AG_{\mu}=\text{GKdim }A^{G,\mu}$. Let $M$ be a f.g\ right
$AG_{\mu}$-module. It must also be f.g.\ as an $A$-module since the extension $A \subset AG_{\mu}$ is
finite by Lemma \ref{lem: fflat}. By \cite[Lemma 5.4]{ardakov2007primeness} it is clear
that the grades of $M_{AG_{\mu}}$ and $M_A$ are equal. One can then apply \cite[Lemma
1.6]{lorenz1988on} to conclude that $\text{GKdim }M_{AG_{\mu}}=\text{GKdim }M_{A}$. 

Piecing this together, we find that
\begin{gather}
\begin{aligned}\label{eq: cohenaagmu}
\text{GKdim }M_{AG_{\mu}}+ j_{AG_{\mu}}(M)= \text{GKdim }M_A+j_A(M) &=\text{GKdim }A \\ &=\text{GKdim }AG_{\mu},
\end{aligned}
\end{gather}
and therefore $AG_{\mu}$ is Cohen-Macaulay. 

Now let $M$ be a f.g.\ right $A^{G,\mu}$-module. By applying a right-sided version of Proposition \ref{prop: brownlevass} with $R=X=A^{G,\mu}$ and $S=AG_{\mu}$ we obtain
\begin{equation*}
AG_{\mu} \otimes_{A^{G,\mu}} \text{Ext}_{A^{G,\mu}}^i \left(M,A^{G,\mu}\right) \cong \text{Ext}_{AG_{\mu}}^i\left(M \otimes_{A^{G,\mu}} AG_{\mu},AG_{\mu}\right).
\end{equation*}
When combined with faithful flatness of the extension $A^{G,\mu} \subset AG_{\mu}$ by Lemma \ref{lem: fflat}, this implies that
\begin{equation*}
j_{A^{G,\mu}}(M)=j_{AG_{\mu}}(M \otimes_{A^{G,\mu}} AG_{\mu}). 
\end{equation*}

By faithful flatness of the extension $A^{G,\mu} \subset AG_{\mu}$, $M$ is contained in $M \otimes_{A^{G,\mu}}
AG_{\mu}$. Therefore by the definition of GK dimension one has 
\begin{equation}\label{eq: gkineq1}
\text{GKdim }M_{A^{G,\mu}} \leq \text{GKdim }(M \otimes_{A^{G,\mu}} AG_{\mu})_{A^{G,\mu}}. 
\end{equation}
By \cite[Proposition 5.6]{krause2000growth} one has the inequality 
\begin{equation}\label{eq: gkineq2}
\text{GKdim }M_{A^{G,\mu}} \geq \text{GKdim }(M \otimes_{A^{G,\mu}} AG_{\mu})_{AG_{\mu}}.
\end{equation}

Applying \cite[Lemma 1.6]{lorenz1988on} to $M \otimes_{A^{G,\mu}} AG_{\mu}$ and using the inequalities in
\eqref{eq: gkineq1} and \eqref{eq: gkineq2} shows that $\text{GKdim }M_{A^{G,\mu}}=\text{GKdim }(M \otimes_{A^{G,\mu}}
AG_{\mu})_{AG_{\mu}}$. One can then see from an equality like that in \eqref{eq: cohenaagmu} that the
Cohen-Macaulay property is preserved under cocycle twists.

(ii) Using \cite[Proposition 3.9(i)]{yi1995injective} one can see that if $A$ satisfies the Auslander condition then so must $AG_{\mu}$. The twist $A^{G,\mu}$ then satisfies the Auslander condition by \cite[Theorem 2.2(iv)]{teo1996homological}, since the only hypothesis needed is that the extension be flat -- this is true by Lemma \ref{lem: fflat}. 

It remains to show that injective dimension is preserved and, by symmetry it suffices to show this for left injective dimension. Consider the $G$-grading on $AG_{\mu}$ for which $(AG_{\mu})_g = A \otimes g$ for all $g \in G$. Under this grading $AG_{\mu}$ is a strongly $G$-graded ring, thus one can apply \cite[Corollary 2.7]{nastasescu1983strongly} with $R = N = AG_{\mu}$ and $\sigma = e$. That result implies that
\begin{equation*}
\text{idim }_{AG_{\mu}}AG_{\mu} = \text{idim }_{(AG_{\mu})_{e}}(AG_{\mu})_{e}= \text{idim }_{A}A.
\end{equation*}

Now consider the $G$-grading on $AG_{\mu}$ under which $(AG_{\mu})_g = A^{G,\mu}(1 \otimes g)$ for all $g \in G$. This $G$-grading is induced by the diagonal action of $G$ on $AG_{\mu}$. It is clear that $AG_{\mu}$ is a strongly $G$-graded ring under this grading as well. One can therefore apply \cite[Corollary 2.7]{nastasescu1983strongly} once again to see that 
\begin{equation*}
\text{idim }_{AG_{\mu}}AG_{\mu} = \text{idim }_{(AG_{\mu})_{e}}(AG_{\mu})_{e}= \text{idim }_{A^{G,\mu}}A^{G,\mu}.
\end{equation*}

(iii) We saw in the proof of (ii) that the Auslander condition is preserved. One can then use Proposition \ref{prop: gldim} to show that the global dimensions of $A$ and $A^{G,\mu}$ are equal, which completes the proof.
\end{proof}

\section{Twists in relation to Rogalski and Zhang's classification}\label{sec: rogzhang}
We will now apply the prior theory to the work in \cite{rogalski2012regular}. As such, we will assume that $\text{char}(k)=0$ for the duration of this section.

Rogalski and Zhang classify AS-regular domains of dimension 4 satisfying two extra conditions: they are generated by three degree 1 elements and admit a proper $\Z^{2}$-grading. Properness of such a grading, $A= \bigoplus_{n,m \in \Z} A_{m,n}$ say,
means that $A_{0,1}\neq 0$ and $A_{1,0}\neq 0$. Their main results are summarised in the following theorem.

\begin{theorem}[{\cite[Theorems 0.1 and 0.2]{rogalski2012regular}}]\label{theorem: rogzhangmain}
Let $A$ be an AS-regular domain of dimension 4 which is generated by three degree 1 elements and properly
$\Z^{2}$-graded. Then either $A$ is a normal extension of an AS-regular algebra of dimension 3, or up to isomorphism it falls into one of eight 1 or 2 parameter families, $\mathcal{A}-\mathcal{H}$. Moreover, any such algebra is strongly noetherian, Auslander regular and Cohen-Macaulay.
\end{theorem}

In order to study cocycle twists of such algebras we require graded algebra automorphisms, where in this case graded refers to the c.g.\ structure rather than the additional $\Z^{2}$-grading. Section 5 of Rogalski and Zhang's paper is concerned with precisely this topic. The key result is the following, where generic means avoiding some finite set of parameters given in the statement of \cite[Lemma 5.1]{rogalski2012regular}.
\begin{theorem}[{\cite[Theorem 5.2(a)]{rogalski2012regular}}]\label{theorem: rogzhangauts}
Consider a generic AS-regular algebra $A$ in one of the families $\mathcal{A} - \mathcal{H}$. The graded automorphism group of $A$ is isomorphic either to $k^{\times}\times k^{\times}$ or to $k^{\times}\times
k^{\times} \times C_2$. The first case occurs for the families $\mathcal{A}(b,q)$ with $q \neq -1$, $\mathcal{D}(h,b)$ with $h
\neq b^4$, $\mathcal{F}$ and $\mathcal{H}$. The second case occurs if $A$ belongs to one of the families $\mathcal{A}(b,-1)$, $\mathcal{B}$, $\mathcal{C}$, $\mathcal{D}(h,b)$ with $h=b^4$, $\mathcal{E}$ or $\mathcal{G}$.
\end{theorem}
The automorphisms corresponding to $k^{\times}\times k^{\times}$ come from scaling components of the $\Z^{2}$-grading.

Let us fix some notation for the remainder of the section. The algebra $A$ will be generated by the three degree 1 elements $x_1,x_2$ and $x_3$, where $x_1, x_2 \in A_{1,0}$ and $x_3 \in A_{0,1}$. We will follow Rogalski and Zhang in referring to the extra automorphism of order 2 as the \emph{quasi-trivial} automorphism. This automorphism interchanges $x_1$ and $x_2$ whilst fixing $x_3$.

We now move onto the main result of this section, in which we show that the presence of the quasi-trivial automorphism implies the existence of cocycle twists relating algebras in different families. We will use the notation $v_i$ to denote generators of a cocycle twist when we wish to suppress the new multiplication symbol $\ast_{\mu}$.
\begin{theorem}\label{theorem: rogzhangmymain}
Let $G=(C_2)^2=\langle g_1, g_2\rangle$ and let $\mu$ denote the 2-cocycle on $G$ defined by 
\begin{equation*}
\mu(g_1^p g_2^q, g_1^r g_2^s) = (-1)^{ps}
\end{equation*}
for all $p,q,r,s \in \{0,1\}$. Fix the isomorphism $G \cong G^{\vee}$ given by $g \mapsto \chi_g$, where 
\begin{equation*}
\chi_g(h) = \left\{ \begin{array}{cl} 1 & \text{if }g=e\text{ or }h \in \{e, g\} \\ -1 & \text{otherwise}, \end{array}\right.
\end{equation*}
for all $g, h \in G$. Then there are $k$-algebra isomorphisms
\begin{align*}
\mathcal{A}(1,-1)^{G,\mu}\cong \mathcal{D}(1,1),\;\; \mathcal{B}(1)^{G,\mu} &\cong \mathcal{C}(1), \;\;
\mathcal{E}(1,\gamma)^{G,\mu}\cong \mathcal{E}(1,-\gamma), \\ 
\mathcal{G}(1,\gamma)^{G,\mu} &\cong \mathcal{G}(1,\overline{\gamma}).
\end{align*}
\end{theorem}

In order to save space we will not write out the defining relations of these algebras, but recommend that the reader has \cite{rogalski2012regular} to refer to when reading the proof.
\begin{proofof}{\ref{theorem: rogzhangmymain}}
Let us begin by defining the action of $G$ which we will use for each of the cocycle twists we perform. Note that all of the algebras in the statement of the result admit the quasi-trivial automorphism. Therefore we can let $g_1$ act via the quasi-trivial automorphism and $g_2$ act by multiplying $x_3$ by -1 and fixing the other two generators. 

Since the standard generators are not diagonal with respect to this action, we will instead use the generators
 \begin{equation*}
 w_1=x_1+x_2,\;\; w_2=x_1-x_2,\;\; w_3=x_3,
 \end{equation*}
which are homogeneous with respect to any induced $G$-grading, since all automorphisms act on them diagonally. Denoting
the algebra we wish to twist by $A$, the induced $G$-grading on the new generators is given by
\begin{equation*}
w_1 \in A_e,\;\; w_2 \in A_{g_{2}},\;\; w_3 \in A_{g_{1}}. 
\end{equation*}

The defining relations of any algebra in one of the eight families belong to different components of the $\Z^2$-grading. Observe that the algebras $\mathcal{A}(1,-1)$, $\mathcal{B}(1)$, $\mathcal{C}(1)$ and $\mathcal{D}(1,1)$ share three relations, only being distinguished from each other by their relations in the $(2,1)$-component. Writing the shared relations in terms of the diagonal basis we show that they are left invariant under the twist. 

The first two are quadratic relations:
\begin{align*}
0 &= w_1^2 - w_2^2 = \frac{w_1 \ast_{\mu} w_1}{\mu(e,e)} - \frac{w_2 \ast_{\mu} w_2}{\mu(g_2,g_2)} = w_1 \ast_{\mu} w_1 -
w_2 \ast_{\mu} w_2 = v_1^2 - v_2^2,\\
0 &= w_3 w_1 - w_1 w_3 = \frac{w_3 \ast_{\mu} w_1}{\mu(g_1,e)} - \frac{w_1 \ast_{\mu} w_3}{\mu(e,g_1)} = w_3 \ast_{\mu}
w_1 - w_1 \ast_{\mu} w_3 = v_3 v_1 - v_1 v_3,
\end{align*}
while the third relation is cubic:
\begin{align*}
0 = w_3^2w_2 - w_2 w_3^2 &= \frac{w_3 \ast_{\mu} w_3 \ast_{\mu} w_2}{\mu(g_1,g_1)\mu(e,g_2)} - \frac{w_2 \ast_{\mu} w_3
\ast_{\mu} w_3}{\mu(g_2,g_1)\mu(g_1g_2,g_1)} \\
&= w_3 \ast_{\mu} w_3 \ast_{\mu} w_2 - w_2 \ast_{\mu} w_3 \ast_{\mu} w_3 \\ &= v_3^2 v_2 - v_2 v_3^2.
\end{align*}

Thus, to verify the first two isomorphisms in the statement of the result it suffices to consider the behaviour under the twist of the only relation they do not share. We first twist this relation in the algebra $\mathcal{A}(1,-1)$, having once again written it in terms of the new generators beforehand:
\begin{align*}
0 &= [w_3,[w_1,w_2]_+] \\ 
&= w_3 w_1 w_2 + w_3 w_2 w_1 - w_1 w_2 w_3 - w_2 w_1 w_3 \\
&= \frac{w_3 \ast_{\mu} w_1 \ast_{\mu} w_2}{\mu(g_1,e)\mu(g_1,g_2)} + \frac{w_3 \ast_{\mu} w_2 \ast_{\mu}
w_1}{\mu(g_1,g_2)\mu(g_1g_2,e)} - \frac{w_1 \ast_{\mu} w_2 \ast_{\mu} w_3}{\mu(e,g_2)\mu(g_2,g_1)} - \frac{w_2
\ast_{\mu} w_1 \ast_{\mu} w_3}{\mu(g_2,e)\mu(g_2,g_1)}  \\
&= -w_3 \ast_{\mu} w_1 \ast_{\mu} w_2 - w_3 \ast_{\mu} w_2 \ast_{\mu} w_1 - w_1 \ast_{\mu} w_2 \ast_{\mu} w_3 - w_2
\ast_{\mu} w_1 \ast_{\mu} w_3 \\
&= -[v_3,[v_1,v_2]_+]_+.
\end{align*}
This relation is the same as that in $\mathcal{D}(1,1)$ under the new generators, which proves the first isomorphism. 

Let us now move on to $\mathcal{B}(1)$. Twisting the non-shared relation we see that
\begin{align*}
0 &= [w_3,[w_2,w_1]]_+ \\ 
&= w_3 w_2 w_1 - w_3 w_1 w_2 + w_2 w_1 w_3 - w_1 w_2 w_3 \\
&= \frac{w_3 \ast_{\mu} w_2 \ast_{\mu} w_1}{\mu(g_1,g_2)\mu(g_1g_2,e)} - \frac{w_3 \ast_{\mu} w_1 \ast_{\mu}
w_2}{\mu(g_1,e)\mu(g_1,g_2)} + \frac{w_2 \ast_{\mu} w_1 \ast_{\mu} w_3}{\mu(g_2,e)\mu(g_2,g_1)} - \frac{w_1 \ast_{\mu}
w_2 \ast_{\mu} w_3}{\mu(e,g_2)\mu(g_2,g_1)} \\
&= -w_3 \ast_{\mu} w_2 \ast_{\mu} w_1 + w_3 \ast_{\mu} w_1 \ast_{\mu} w_2 + w_2 \ast_{\mu} w_1 \ast_{\mu} w_3 - w_1
\ast_{\mu} w_2 \ast_{\mu} w_3 \\
&= [v_3,[v_1,v_2]].
\end{align*}
This relation is shared by $\mathcal{C}(1)$ under the new generating set, which proves the second isomorphism.

We now move on to the remaining two isomorphisms. The algebras in the relevant families share three relations, two of which we have already shown are preserved under the cocycle twist. This is also true for the third relation, which as yet we have not encountered:
\begin{align*}
0 = w_3^2w_2 +w_2w_3^2 &= \frac{w_3 \ast_{\mu} w_3 \ast_{\mu} w_2}{\mu(g_1,g_1)\mu(e,g_2)} + \frac{w_2 \ast_{\mu} w_3
\ast_{\mu} w_3}{\mu(g_2,g_1)\mu(g_1g_2,g_1)} \\&= w_3 \ast_{\mu} w_3 \ast_{\mu} w_2 + w_2 \ast_{\mu} w_3 \ast_{\mu} w_3 \\ &= v_3^2 v_2 +v_2 v_3^2.
\end{align*}

Once again, it suffices to see what happens to the non-shared relation. In $\mathcal{E}(1,\gamma)$, where $i = \sqrt{-1}$ and $\gamma = \pm i$, one has
\begin{align*}
0 &= w_3w_2w_1 - w_1w_3w_2 +\gamma w_1w_2w_3 - \gamma w_2w_1w_3 \\
&= \frac{w_3 \ast_{\mu} w_2 \ast_{\mu} w_1}{\mu(g_1,g_2)\mu(g_1g_2,e)} - \frac{w_1 \ast_{\mu} w_3 \ast_{\mu}
w_2}{\mu(e,g_1)\mu(g_1,g_2)} + \gamma \frac{w_1 \ast_{\mu} w_2 \ast_{\mu} w_3}{\mu(e,g_2)\mu(g_2,g_1)} - \gamma \frac{w_2 \ast_{\mu} w_1 \ast_{\mu} w_3}{\mu(g_2,e)\mu(g_2,g_1)} \\
&= -w_3 \ast_{\mu} w_2 \ast_{\mu} w_1 + w_1 \ast_{\mu} w_3 \ast_{\mu} w_2 + \gamma w_1 \ast_{\mu} w_2 \ast_{\mu} w_3 -
\gamma w_2 \ast_{\mu} w_1 \ast_{\mu} w_3 \\
&= -v_3v_2v_1 + v_1v_3v_2 +\gamma v_1v_2v_3 - \gamma v_2v_1v_3.
\end{align*}
This is the final relation in $\mathcal{E}(1,-\gamma)$ under the new generators, which proves the penultimate isomorphism.

We now twist the final relation of $\mathcal{G}(1,\gamma)$, where $\gamma=\frac{1 + i}{2}$ and so $\overline{\gamma}=\frac{1}{2 \gamma}$:
\begin{align*}
0 &= w_3 w_1 w_2 +w_3 w_2 w_1 +i w_1w_2w_3 + i w_2w_1w_3 \\
&= \frac{w_3 \ast_{\mu} w_1 \ast_{\mu} w_2}{\mu(g_1,e)\mu(g_1,g_2)} + \frac{w_3 \ast_{\mu} w_2 \ast_{\mu}
w_1}{\mu(g_1,g_2)\mu(g_1g_2,e)} +i\frac{w_1 \ast_{\mu} w_2 \ast_{\mu} w_3}{\mu(e,g_2)\mu(g_2,g_1)} + i \frac{w_2
\ast_{\mu} w_1 \ast_{\mu} w_3}{\mu(g_2,e)\mu(g_2,g_1)} \\
&= -w_3 \ast_{\mu} w_1 \ast_{\mu} w_2 - w_3 \ast_{\mu} w_2 \ast_{\mu} w_1 + i w_1 \ast_{\mu} w_2 \ast_{\mu}
w_3 + i w_2 \ast_{\mu} w_1 \ast_{\mu} w_3 \\
&= -v_3  v_1  v_2 - v_3  v_2  v_1 + i v_1  v_2  v_3 + i v_2  v_1  v_3.
\end{align*}
This is precisely the final relation of $\mathcal{G}(1,\overline{\gamma})$ under the new generators, which proves the last isomorphism in the statement of the theorem.
\end{proofof}

Combined with the fact that $\mathcal{A}(b,-1)$, $\mathcal{B}(b)$, $\mathcal{C}(b)$, $\mathcal{D}(b,b^4)$,
$\mathcal{E}(b,\gamma)$ and $\mathcal{G}(b,\gamma)$ are Zhang twists of the respective algebras in the statement of Theorem \ref{theorem: rogzhangmymain} for any parameter $b \in k^{\times}$ \cite[\S 3]{rogalski2012regular}, this result gives further information about such algebras up to Zhang twists. 

\bibliographystyle{abbrv}
\bibliography{allrefs}
\end{document}